\documentclass[11pt,a4paper]{article}
\usepackage[utf8]{inputenc}
\usepackage{amsmath}
\usepackage{amsfonts}
\usepackage{amssymb}
\usepackage{amsthm}

\usepackage{tikz}
\usepackage[font={small,it}]{caption}

\newtheorem{prop}{Proposition}
\newtheorem{mylem}[prop]{Lemma}
\newtheorem{mycor}[prop]{Corollary}
\newtheorem{mytheo}[prop]{Theorem}

\theoremstyle{definition}
\newtheorem{mydef}[prop]{Definition}

\usepackage{graphicx}

\numberwithin{equation}{section}
\numberwithin{prop}{section}

\DeclareMathOperator{\Sym}{Sym}

\title{Graphs of kei and their diameters}
\author{Matthew Ashford\thanks{Mathematical Institute, University of Oxford, Radcliffe
Observatory Quarter, Woodstock Road, Oxford OX2 6GG, UK. E-mail: \texttt{matthew.ashford1@btinternet.com}.}}

\begin{document}
\maketitle

\begin{abstract}
A kei on $[n]$ can be thought of as a set of maps $(f_x)_{x \in [n]}$, where each $f_x$ is an involution on $[n]$ such that $(x)f_x = x$ for all $x$ and $f_{(x)f_y} = f_yf_xf_y$ for all $x$ and $y$. We can think of kei as loopless, edge-coloured multigraphs on $[n]$ where we have an edge of colour $y$ between $x$ and $z$ if and only if $(x)f_y = z$; in this paper we show that any component of diameter $d$ in such a graph must have at least $2^d$ vertices and contain at least $2^{d-1}$ edges of the same colour. We also show that these bounds are tight for each value of $d$.
\end{abstract}

\section{Introduction}

A \emph{kei} (or \emph{involutive quandle}) is a pair $(X, \triangleright)$, where $X$ is a non-empty set and $\triangleright:X \times X \to X$ is a binary operation such that:
\begin{enumerate}
\item For any $y, z \in X$, there exists $x \in X$ such that $z = x \triangleright y$;
\item Whenever we have $x,y,z \in X$ such that $x \triangleright y = z \triangleright y$, then $x=z$;
\item For any $x,y,z \in X$, $(x \triangleright y) \triangleright z = (x \triangleright z) \triangleright (y \triangleright z)$;
\item For any $x \in X$, $x \triangleright x = x$;
\item For any $x,y \in X$, $(x \triangleright y) \triangleright y = x$.
\end{enumerate}
Note that conditions 1 and 2 above are equivalent to the statement that for each $y$, the map $x \mapsto x \triangleright y$ is a bijection on $X$.

A \emph{quandle} is a pair $(X, \triangleright)$ satisfying conditions 1--4 above, while a \emph{rack} is a pair $(X, \triangleright)$ satisfying conditions 1--3. As mentioned in \cite{mainpaper}, racks originally developed from correspondence between J.H. Conway and G.C. Wraith in 1959, while quandles were introduced independently by Joyce \cite{quanhist1} and Matveev \cite{quanhist2} in 1982 as invariants of knots, and kei were first studied by Takasaki \cite{keihist} in 1943. Fenn and Rourke \cite{history} provide a history of racks and quandles, while Nelson \cite{widerhist} gives an overview of how these structures relate to other areas of mathematics; a recent paper by Stanovsk{\`y} \cite{survey} gives a thorough survey of the history of research on kei.

As a first example of a kei, note that for any set $X$, if we define $x \triangleright y = x$ for all $x,y \in X$, we obtain a kei known as the \emph{trivial kei} $T_X$. Let $G$ be a group and let $X$ be the set of all involutions of $G$. If we define a binary operation $\triangleright:X \times X \to X$ by $x \triangleright y := y^{-1}xy$, then $(X, \triangleright)$ is a kei; it is an example of a \emph{conjugation quandle}. For a further example, define a binary operation on $[n]$ by setting $i \triangleright j := 2j - i$ (mod $n$); $([n], \triangleright)$ is known as a \emph{dihedral kei}.

For any kei $(X, \triangleright)$, we can define a set of involutions $(f_y)_{y \in X}$ by setting $(x)f_y = x \triangleright y$ for all $x$ and $y$. The following well-known result (see for example, \cite{history}, \cite{mainpaper}) gives the correct conditions for a collection of maps $(f_y)_{y \in X}$ to define a kei.
\begin{prop}
\label{mapdef}
Let $X$ be a set and $(f_x)_{x \in X}$ be a collection of functions each with domain and co-domain $X$. Define a binary operation $\triangleright:X \times X \to X$ by $x \triangleright y := (x)f_y$. Then $(X, \triangleright)$ is a kei if and only if $f_y$ is an involution for each $y \in X$ and the following conditions hold: for all $y,z \in X$ we have
\begin{equation}
\label{maps}
f_{(y)f_z} = f_zf_yf_z,
\end{equation}
and for all $x \in X$ we have
\begin{equation}
\label{quancond}
(x)f_x = x.
\end{equation}
\end{prop}
\begin{proof}
As noted earlier, each $f_y$ is a bijection; it remains to show that items 3 and 4 in the definition of a kei are equivalent to \eqref{maps} and \eqref{quancond} respectively, while item 5 is equivalent to the statement that each $f_y$ is an involution. This is essentially a reworking of the definition and we omit the simple details.
\end{proof}
This means that we can just as well define a kei on a set $X$ by the set of maps $(f_y)_{y \in X}$, providing they are all involutions satisfying \eqref{maps} and \eqref{quancond}. We will move freely between the two definitions, with $x \triangleright y = (x)f_y$ for all $x,y \in X$ unless otherwise stated.

Now observe that any kei on $X$ can be represented by a multigraph on $X$; we give each vertex a colour and then put an edge of colour $i$ from vertex $j$ to vertex $k$ if and only if $(j)f_i = k$. This is well-defined as each $f_i$ is an involution, so $(j)f_i = k$ if and only if $(k)f_i = j$. We then remove all loops from the graph; i.e.\ if $(j)f_i = j$ we don't have an edge of colour $i$ incident to $j$.

It will be helpful to recast the representation of kei by multigraphs in a slightly different setting. Let $V$ be a finite set and let $\sigma \in \Sym(V)$ be an involution; then we can define a simple graph $G_\sigma$ on $V$ by letting $uv \in E(G_\sigma)$ if and only if $u \neq v$ and $(u)\sigma = v$. As $\sigma$ is a disjoint product of transpositions, we see that $G_\sigma$ consists of a partial matching and some isolated vertices. We can now extend this definition to the case of multiple involutions in a natural way.
\begin{mydef}
\label{multinvodef}
Suppose $\Sigma = \lbrace \sigma_1, \ldots, \sigma_k \rbrace \subseteq \Sym(V)$ is a set of involutions on a set $V$. Define a loopless multigraph $G_\Sigma = (V,E)$ with a $k$-edge-colouring by putting an edge of colour $i$ from $u$ to $v$ if and only if $u \neq v$ and $(u)\sigma_i = v$.

We also define the \emph{reduced} graph $G_\Sigma^0$ to be the simple graph on $V$ obtained by setting $e = uv \in E(G_\Sigma^0)$ if and only if there is at least one edge from $u$ to $v$ in $G_\Sigma$. 
\end{mydef}
Observe that if $\Sigma' \subseteq \Sigma$, then $G_{\Sigma'}$ is a subgraph of $G_\Sigma$. Now let us return specifically to kei.
\begin{mydef}
Let $K = (X, \triangleright)$ be a kei, and let $(f_y)_{y \in X}$ be the associated maps. For any $S \subseteq X$, define $\Sigma_S = \lbrace f_y \mid y \in S \rbrace$. Then by $G_S$ we mean the multigraph $G_{\Sigma_S}$ in the sense of Definition \ref{multinvodef}; $G_S$ thus has an associated $|S|$-edge-colouring, although if $|S|=1$ we may not necessarily consider $G_S$ as being coloured. We will also write $G_K = G_X$, indicating the graph for the whole kei.
\end{mydef}
Figure \ref{examples} gives two examples of graphs representing kei.
\begin{figure}
\begin{center}
\begin{tikzpicture}[scale=1.4]

\draw [red] (-4,0.336) -- (-3,0.336);
\draw [blue] (-3.5,1.202) -- (-3,0.366);
\draw [green] (-4,0.336) -- (-3.5,1.202);

\draw [fill, green] (-3,0.336) circle [radius=0.1];
\draw [fill, blue] (-4,0.336) circle [radius=0.1];
\draw [fill, red] (-3.5,1.202) circle [radius=0.1]; 
\draw [fill, violet] (-2.25,0.969) circle [radius=0.1];

\node [above right] at (-3.5,1.202) {(1 2)};
\node [below left] at (-4,0.366) {(1 3)};
\node [below right] at (-3,0.336) {(2 3)};
\node [left] at (-2.3,0.969) {$\iota$};

\draw [red] (0,0.1) -- (0.76,0.75);
\draw [red] (3.52,0.1) -- (2.76,0.75);
\draw [blue] (0,1.4) -- (0.76,0.75);
\draw [blue] (3.52,1.4) -- (2.76,0.75);
\draw [green] (0,1.4) -- (0,0.1);
\draw [green] (3.52,1.4) -- (3.52,0.1);
\draw [violet] (0,0.1) -- (2.76,0.75);
\draw [violet] (0.76,0.75) -- (3.52,0.1);
\draw [brown] (0,1.4) -- (2.76,0.75);
\draw [brown] (3.52,1.4) -- (0.76,0.75);
\draw [darkgray] (0,1.4) -- (3.52,0.1);
\draw [darkgray] (3.52,1.4) -- (0,0.1);

\draw [fill, blue] (0,0.1) circle [radius=0.1]; 
\draw [fill, red] (0,1.4) circle [radius=0.1]; 
\draw [fill, green] (0.76,0.75) circle [radius=0.1]; 
\draw [fill, darkgray] (2.76,0.75) circle [radius=0.1]; 
\draw [fill, brown] (3.52,0.1) circle [radius=0.1]; 
\draw [fill, violet] (3.52,1.4) circle [radius=0.1]; 

\node [above left] at (0,1.4) {(1 2)};
\node [below left] at (0,0.1) {(1 3)};
\node [left] at (0.755,0.75) {(2 3)};
\node [right] at (2.765,0.75) {(1 4)};
\node [above right] at (3.52,1.4) {(3 4)};
\node [below right] at (3.52,0.1) {(2 3)};
\end{tikzpicture}
\end{center}
\caption{Graphical representations of two kei. Both of these are subquandles of conjugation quandles; $S_3$ on the left and $S_4$ on the right.}
\label{examples}
\end{figure}
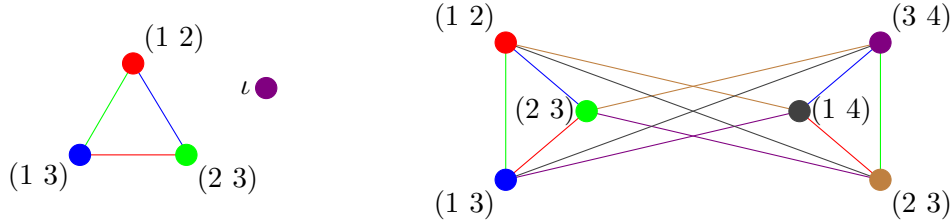
Before stating our main result, we will need some more definitions. Let $(X, \triangleright)$ be a kei; then a \emph{subkei} of $(X, \triangleright)$ is a kei $(Y, \triangleright|_{Y \times Y})$\footnote{The notation $\triangleright|_{Y \times Y}$ in the above context will always be abbreviated to $\triangleright$, with the restriction to the subset $Y$ left implicit.} where $Y \subseteq X$. Thus a subset $Y \subseteq X$ forms a subkei if and only if for all $y,z \in Y$, $(z)f_y \in Y$. For any $T \subseteq X$ and $u,v \in X$, denote by $d_T(u,v)$ the graph distance between $u$ and $v$ in the graph $G_T$. As $G_T^0$ is the simple graph on $X$ formed by ignoring all colours and multiple edges, $d_T(u,v)$ is clearly identical to the graph distance between $u$ and $v$ in the reduced graph $G_T^0$. We will prove the following result.
\begin{mytheo}
\label{diam}
Let $(X, \triangleright)$ be a kei, and let $(S,\triangleright)$ be a subkei. Let $C \subseteq X$ span a component of $G_S$ (or equivalently of $G_S^0$) and suppose that $G_S[C]$ has diameter $d$. Then $|C| \geqslant 2^d$, and further there exists some $k \in S$ such that there are at least $2^{d-1}$ edges of colour $k$ in $G_S[C]$.
\end{mytheo}
The remainder of the paper is organised as follows. In Section \ref{spsec} we show how to construct a large number of shortest paths between two vertices $u$ and $v$ that are connected in $G_S$. In Section \ref{elsec} we will show that any sequence of colours occurring in order (from $u$ to $v$) on a shortest path corresponds to a different vertex connected to $u$, which will prove the result. In Section \ref{extrsec} we give examples to show that Theorem \ref{diam} is tight for all values of $d$.

\section{Shortest paths}
\label{spsec}

We begin by showing that there are many shortest paths between pairs of vertices.
\begin{mylem}
\label{hang}
Let $S \subseteq X$ be a subkei, and let $u,v \in X$ such that $d_S(u,v) = d \in [2, \infty)$; let $u=v_0v_1 \cdots v_d=v$ be a path of length $d$ in $G_S$ between $u$ and $v$. For each $i$, let $c_i$ be the colour of the $v_{i-1}v_i$ edge on the path, so $(v_i)f_{c_i} = v_{i-1}$. Now fix some $i \geqslant 2$ and define $w_k = (v_k)f_{c_i}$ for $k = 0, \ldots, i-2$. Then the vertices $w_0, \ldots, w_{i-2},v_0, \ldots, v_d$ are all distinct and $v_0w_0 \cdots w_{i-2}v_i \cdots v_d$ is a (shortest) $uv$-path in $G_S$.
\end{mylem}
\begin{proof}
(See Figure \ref{shortpathfig} throughout). For any $j \leqslant i-2$, denote by $P(j)$ the statement that the vertices $w_j, \ldots, w_{i-2}, v_0, \ldots, v_d$ are all distinct, and that the path $v_0 \cdots v_jw_j \cdots w_{i-2}v_i \cdots v_d$ is a (shortest) $uv$-path in $G_S$. Then the lemma is the statement $P(0)$; we will prove that $P(j)$ holds for all $j$ by reverse induction. 

\begin{figure}
\begin{center}
\begin{tikzpicture}
\draw (0,1.5) --(4.5,1.5);
\draw [red] (4.5,1.5) -- (6,1.5);
\draw (6,1.5) -- (9,1.5);

\draw [red] (0,0) -- (0,1.5);
\draw [red] (1.5,0) -- (1.5,1.5);
\draw [red] (3,0) -- (3,1.5);

\draw [darkgray] (0,0) -- (3,0) -- (6,1.5);

\draw [fill] (0,1.5) circle [radius=0.1];
\draw [fill] (1.5,1.5) circle [radius=0.1];
\draw [fill] (3,1.5) circle [radius=0.1];
\draw [fill] (4.5,1.5) circle [radius=0.1];
\draw [fill] (6,1.5) circle [radius=0.1];
\draw [fill] (7.5,1.5) circle [radius=0.1];
\draw [fill] (9,1.5) circle [radius=0.1];

\draw [fill] (0,0) circle [radius=0.1];
\draw [fill] (1.5,0) circle [radius=0.1];
\draw [fill] (3,0) circle [radius=0.1];

\node [above] at (6,1.5) {$v_i$};
\node [above] at (4.5,1.5) {$v_{i-1}$};
\node [above] at (3,1.5) {$v_{i-2}$};
\node [above] at (0,1.5) {$v_0$};
\node [below] at (0,-0.1) {$w_0$};
\node [below] at (3,-0.1) {$w_{i-2}$};
\node [above] at (9,1.5) {$v_d$};
\end{tikzpicture}
\end{center}
\caption{An example with $i=4$. The colour $c_i$ is represented by red, with the colours on the lower path being the colours $c'$, $c''$ etc. found as we progress from right to left along the upper path.}
\label{shortpathfig}
\end{figure}
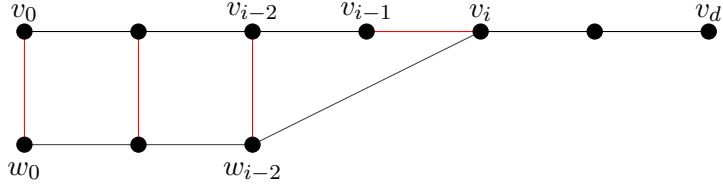
Firstly, as $w_{i-2} = (v_{i-2})f_{c_i}$ and $v_i = (v_{i-1})f_{c_i}$ we see that $w_{i-2} \notin \lbrace v_{i-1}, v_i \rbrace$ (or there would be a vertex with two incident edges of the same colour). If we had $w_{i-2} = v_l$ for some $l > i$ then $v_0 \cdots v_{i-2}v_l \cdots v_d$ would be a $uv$-path in $G_S$ of length $d - (l-i+1) < d$, a contradiction. Now observe that
\[ w_{i-2} := (v_{i-2})f_{c_i} = (v_{i-1})f_{c_{i-1}}f_{c_i} = (v_i)f_{c_i}f_{c_{i-1}}f_{c_i} = (v_i)f_{c'}, \]
where $c' = (c_{i-1})f_{c_i}$; note that $c' \in S$ as $S$ is a subkei and so there is an edge in $G_S$ between $v_i$ and $w_{i-2}$. So now suppose $w_{i-2} = v_l$ for $l < i-1$; then $v_0 \cdots v_l=w_{i-2}v_i \cdots v_d$ is a $uv$-path in $G_S$ of length $d - (i-l-1)$, again a contradiction. So $w_{i-2} \neq v_l$ for any $l$, and we can use the $v_{i-2}w_{i-2}$ and $w_{i-2}v_i$ edges, of colours $c_i$ and $c'$ respectively, to construct a $uv$-path $v_0 \cdots v_{i-2}w_{i-2}v_i \cdots v_d$ of length $d$ in $G_S$. This establishes $P(i-2)$.

So now take $j < i-2$ and suppose $P(j+1)$ holds. As $v_{j+1}w_{j+1}, \ldots, v_{i-2}w_{i-2}$ and $v_{i-1}v_i$ are all edges of colour $c_i$, $w_j \neq v_l$ for $j < l \leqslant i$ and it is also not equal to any other $w_k$. As before, if $w_j = v_l$ for $l > i$ we can construct a $uv$-path in $G_S$ of length $d - (l-i+1)$, a contradiction. Now observe that
\[ w_j := (v_j)f_{c_i} = (v_{j+1})f_{c_{j+1}}f_{c_i} = (w_{j+1})f_{c_i}f_{c_{j+1}}f_{c_i} = (w_{j+1})f_{c''} , \]
where $c'' = (c_{j+1})f_{c_i} \in S$, so there is an edge in $G_S$ between $w_j$ and $w_{j+1}$. Now suppose $w_j = v_l$ for some $l \leqslant j$; then $v_0 \cdots v_l=w_jw_{j+1} \cdots w_{i-2}v_i \cdots v_d$ is a $uv$-path in $G_S$ of length $d - (j-l+1)$, a contradiction (note the existence of the $w_{j+1}v_i$-path follows from the assumed $P(j+1)$). Hence $w_j \neq v_l$ for any $l$, and we can use the $v_jw_j$ and $w_jw_{j+1}$ edges, of colours $c_i$ and $c''$ respectively, to construct a $uv$-path $v_0 \cdots v_jw_j \cdots w_{i-2}v_i \cdots v_d$ of length $d$ in $G_S$. Thus $P(j)$ holds, and so by reverse induction $P(0)$ holds, proving the lemma.
\end{proof}
\begin{mycor}
\label{distinct}
Let $P$ be any shortest $uv$-path in $G_S$. Then no two edges of $P$ have the same colour.
\end{mycor}
\begin{proof}
Suppose  there is a path $u=v_0 \cdots v_d=v$ and a colour $c \in S$ such that $(v_{j-1})f_c = v_j$ and $(v_{i-1})f_c = v_i$ for some $j < i-1$ (we can't have $j = i-1$ as then $v_j$ is incident to two edges of colour $c$). From the lemma, with $c = c_i$, $w_j = (v_j)f_{c_i} \neq v_l$ for any $l$, but here $w_j = (v_j)f_c = v_{j-1}$, a contradiction.
\end{proof}

\section{Sequences of elements}
\label{elsec}

Suppose we are considering a shortest path $P:(u = v_0, \ldots, v_d = v)$ in $G_S$. In the light of Corollary \ref{distinct}, and to ease notation, we will assume without loss of generality that $X = [n]$ and $[d] \subseteq S$, and that the edge between $v_{i-1}$ and $v_i$ in $P$ is of colour $i$. Now for each strictly increasing sequence $\mathbf{s} = (a_1, \ldots, a_r)$ of elements from $[d]$, define
\[ u_\mathbf{s} = (u)f_{a_1} \cdots f_{a_r}, \]
where we note $u_\mathbf{s} \in S$ as $S$ is a subkei. Now define $U_0 = \lbrace u \rbrace$ and for $r = 1, \ldots, d$, define
\[ U_r = \lbrace u_\mathbf{s} \mid \mathbf{s} \text{ is strictly increasing, } |\mathbf{s}| = r \rbrace, \]
so in particular $U_d = \lbrace v \rbrace$. Now let $\mathbf{e}_i = (1, \ldots, i)$ for all $i$; then $v_i = (u)f_1 \cdots f_i = u_{\mathbf{e}_i}$ and so $v_i \in U_i$ for all $i$. We also have the property that as $\mathbf{s}$ is increasing, $a_i \geqslant i$ for any $i$, with equality if and only if the subsequence consisting of the first $i$ terms of $\mathbf{s}$ is the canonical sequence $\mathbf{e}_i$.

We will show that any strictly increasing sequence $\mathbf{s}$ can appear at the start of a shortest $uv$-path, and that any such path passes sequentially through each of $U_0, U_1, \ldots, U_d$.
\begin{mylem}
\label{seq}
Let $S \subseteq X$ be a subkei, and let $u,v \in X$ be such that $d_S(u,v) = d \in [2, \infty)$; let $u=v_0v_1 \cdots v_d=v$ be a path of length $d$ in $G_S$ between $u$ and $v$, where the $v_{i-1}v_i$ edge is of colour $i$ for each $i$. Let $\mathbf{s} = (a_1, \ldots, a_r)$ be a strictly increasing sequence of elements from $[d]$. Then there is a shortest path $P_\mathbf{s}:(u=x_0x_1 \cdots x_{a_r - 1}x_{a_r} = v_{a_r} \cdots v_d)$ such that $x_k \in U_k$ for $1 \leqslant k < a_r$ and the $x_{k-1}x_k$ edge in $P_\mathbf{s}$ is of colour $a_k$ for $1 \leqslant k \leqslant r$ (which means in particular that $x_r = u_\mathbf{s}$).
\end{mylem}
\begin{proof}
We shall prove the following stronger statement by induction on $r$: there exists a shortest path $P_\mathbf{s}:(u=x_0x_1 \cdots x_{a_r - 1}x_{a_r} = v_{a_r} \cdots v_d)$ such that for $1 \leqslant k < a_r$, 
\begin{equation}
\label{strongcond}
x_k = u_\mathbf{t}
\end{equation}  
for a (strictly increasing) sequence $\mathbf{t}$ of length $k$ whose largest element is at most $a_r$, and that for $1 \leqslant k \leqslant r$ the $x_{k-1}x_k$ edge in $P_\mathbf{s}$ is of colour $a_k$. These statements clearly imply the result.

First consider the base case $r=1$; if $\mathbf{s} = \mathbf{e}_1 = (1)$ then the original path $P$ suffices as the first edge is of colour 1, so $v_1 = (u)f_1 = u_{(1)}$. Hence we may assume $a_1 > 1$. Then applying Lemma \ref{hang} to $P$ with $i = a_1$, there exists a shortest path $P':(u=v_0w_0 \cdots w_{a_1 - 2}v_{a_1} \cdots v_d)$ in $G_S$, with an edge of colour $a_1$ between $v_k$ and $w_k$ for $0 \leqslant k \leqslant a_1 - 2$; in particular, the first edge of $P'$ is of colour $a_1$. Now for $k \leqslant a_1 - 2$, $w_k = (v_k)f_{a_1} = (u_{\mathbf{e}_k})f_{a_1}$, so $w_k = u_{(\mathbf{e}_k,a_1)} \in U_{k+1}$. So we obtain \eqref{strongcond} by setting $x_k = w_{k-1}$ for $1 \leqslant k < a_1$ and putting $P_\mathbf{s} = P'$.

So now take $r > 1$ and assume the result for smaller $r$. If $a_r = r$ we have $\mathbf{s} = \mathbf{e}_r$, and in this situation the original path $P$ suffices as the first $r$ edges are of colours $1, \ldots, r$ and $v_k \in U_k$ for $1 \leqslant k \leqslant r$. So we can assume that $a_r > r$.

By applying the inductive hypothesis to the sequence $\mathbf{s'} = (a_1, \ldots, a_{r-1})$ we see that there exists a path 
\[ P_\mathbf{s'}:(u=y_0y_1 \cdots y_{a_{r-1}-1}y_{a_{r-1}} = v_{a_{r-1}} \cdots v_d) \]
such that, for $1 \leqslant k < a_{r-1}$, $y_k = u_\mathbf{t}$ for a strictly increasing sequence $\mathbf{t}$ whose largest element is at most $a_{r-1}$. We also have that the $y_{k-1}y_k$ edge is of colour $a_k$ for $1 \leqslant k < r$, but we don't know the colours of the edges $y_{k-1}y_k$ for $r \leqslant k \leqslant a_{r-1}$. But as $a_r > a_{r-1}$ the $v_{a_r -1}v_{a_r}$ edge of colour $a_r$ is still present in $P_\mathbf{s'}$; hence we can apply Lemma \ref{hang} to $P_\mathbf{s'}$, with $i = a_r$, to obtain a shortest path $P'':(u=v_0w_0 \cdots w_{a_r -2}v_{a_r} \cdots v_d)$ in $G_S$, with an edge of colour $a_r$ between $y_k$ and $w_k$ for $0 \leqslant k \leqslant \min \lbrace a_{r-1}, a_r - 2 \rbrace$ (note that $a_r - 2 < a_{r-1}$ if and only if $a_r = a_{r-1} + 1$). We aim to show that for all $0 \leqslant k \leqslant a_r -2$, $w_k = u_\mathbf{t'}$ for some sequence $\mathbf{t'}$ of length $k + 1$ whose largest element is at most $a_r$. The proof will vary slightly depending on whether $a_r = a_{r-1}+1$ (see Figure \ref{closepath}) or $a_r > a_{r-1}+1$ (see Figure \ref{farpath}).

Fix a $k$ such that $0 \leqslant k \leqslant a_{r-1}-1 \leqslant a_r - 2$; by the inductive hypothesis (specifically \eqref{strongcond}) $y_k = u_\mathbf{t}$ for some sequence $\mathbf{t}$ whose $k$th and largest element is at most $a_{r-1}$. But $w_k = (y_k)f_{a_r} = (u_\mathbf{t})f_{a_r}$, so $w_k = u_{(\mathbf{t},a_r)}$; as $a_r > a_{r-1}$ we have $w_k \in U_{k+1}$, where $w_k = u_\mathbf{t'}$ for some sequence $t'$ whose largest element is at most $a_r$. If $a_r = a_{r-1} + 1$ then we have considered all the vertices $w_1, \ldots, w_{a_r - 2}$.

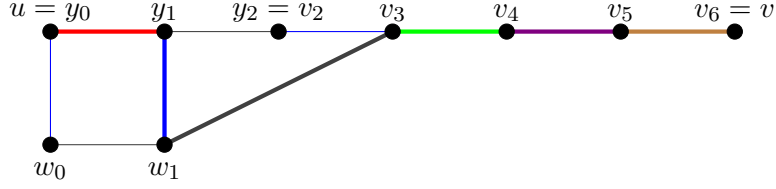
\begin{figure}
\begin{center}
\begin{tikzpicture}
\draw [red, ultra thick] (0,1.5) --(1.5,1.5);
\draw [darkgray] (1.5,1.5) -- (3,1.5);
\draw [blue] (3,1.5) -- (4.5,1.5);
\draw [green, ultra thick] (4.5,1.5) -- (6,1.5);
\draw [violet, ultra thick] (6,1.5) -- (7.5,1.5);
\draw [brown, ultra thick] (7.5,1.5) -- (9,1.5);

\draw [blue] (0,0) -- (0,1.5);
\draw [blue, ultra thick] (1.5,0) -- (1.5,1.5);

\draw [darkgray] (0,0) -- (1.5,0);
\draw [darkgray, ultra thick] (1.5,0) -- (4.5,1.5);

\draw [fill] (0,1.5) circle [radius=0.1];
\draw [fill] (1.5,1.5) circle [radius=0.1];
\draw [fill] (3,1.5) circle [radius=0.1];
\draw [fill] (4.5,1.5) circle [radius=0.1];
\draw [fill] (6,1.5) circle [radius=0.1];
\draw [fill] (7.5,1.5) circle [radius=0.1];
\draw [fill] (9,1.5) circle [radius=0.1];

\draw [fill] (0,0) circle [radius=0.1];
\draw [fill] (1.5,0) circle [radius=0.1];

\node [above] at (9,1.5) {$v_6=v$};
\node [above] at (7.5,1.5) {$v_5$};
\node [above] at (6,1.5) {$v_4$};
\node [above] at (4.5,1.5) {$v_3$};
\node [above] at (3,1.5) {$y_2 = v_2$};
\node [above] at (1.5,1.5) {$y_1$};
\node [above] at (0,1.5) {$u=y_0$};

\node [below] at (0,-0.1) {$w_0$};
\node [below] at (1.5,-0.1) {$w_1$};
\end{tikzpicture}
\end{center}
\caption{Suppose $\mathbf{s} = (2,3)$ and that edges of colour $2$ are red while those of colour $3$ are blue. The top path $P_\mathbf{s'}$ corresponds to the sequence $(2)$, and the bottom path $P''$ is a result of applying Lemma \ref{hang}. The new path $P'''$ is shown in bold.}
\label{closepath}
\end{figure}

Now suppose that $a_r > a_{r-1} + 1$; in this case we also have an edge of colour $a_r$ between $v_k$ and $w_k$ for $a_{r-1} \leqslant k \leqslant a_r - 2$. But as before, $w_k = (v_k)f_{a_r} = (u_{\mathbf{e}_k})f_{a_r} = u_{(\mathbf{e}_k,a_r)}$ for any such $k$, so $w_k \in U_{k+1}$ and $w_k = u_\mathbf{t'}$ for some sequence $t'$ whose largest element is at most $a_r$. Thus we have the desired result on $w_k$ for the entire range $0 \leqslant k \leqslant a_r-2$.

\begin{figure}
\begin{center}
\begin{tikzpicture}
\draw [red, ultra thick] (0,1.5) --(1.5,1.5);
\draw [blue, ultra thick] (1.5,1.5) -- (3,1.5);
\draw [darkgray] (3,1.5) -- (4.5,1.5);
\draw [green] (4.5,1.5) -- (6,1.5);
\draw [violet] (6,1.5) -- (7.5,1.5);
\draw [brown, ultra thick] (7.5,1.5) -- (9,1.5);

\draw [violet] (0,0) -- (0,1.5);
\draw [violet] (1.5,0) -- (1.5,1.5);
\draw [violet, ultra thick] (3,0) -- (3,1.5);
\draw [violet] (4.5,0) -- (4.5,1.5);

\draw [darkgray] (0,0) -- (3,0);
\draw [darkgray, ultra thick] (3,0) -- (4.5,0) -- (7.5,1.5);

\draw [fill] (0,1.5) circle [radius=0.1];
\draw [fill] (1.5,1.5) circle [radius=0.1];
\draw [fill] (3,1.5) circle [radius=0.1];
\draw [fill] (4.5,1.5) circle [radius=0.1];
\draw [fill] (6,1.5) circle [radius=0.1];
\draw [fill] (7.5,1.5) circle [radius=0.1];
\draw [fill] (9,1.5) circle [radius=0.1];

\draw [fill] (0,0) circle [radius=0.1];
\draw [fill] (1.5,0) circle [radius=0.1];
\draw [fill] (3,0) circle [radius=0.1];
\draw [fill] (4.5,0) circle [radius=0.1];

\node [above] at (9,1.5) {$v_6 = v$};
\node [above] at (7.5,1.5) {$v_5$};
\node [above] at (6,1.5) {$v_4$};
\node [above] at (4.5,1.5) {$y_3 = v_3$};
\node [above] at (3,1.5) {$y_2$};
\node [above] at (1.5,1.5) {$y_1$};
\node [above] at (0,1.5) {$u=y_0$};

\node [below] at (0,-0.1) {$w_0$};
\node [below] at (1.5,-0.1) {$w_1$};
\node [below] at (3,-0.1) {$w_2$};
\node [below] at (4.5,-0.1) {$w_3$};
\end{tikzpicture}
\caption{Suppose we now have the sequence $\mathbf{s} = (2,3,5)$, with edges of colour 5 being violet. The top path $P_{s'}$ is the (relabelled) bold path from Figure \ref{closepath}, while the bottom path is again a result of applying Lemma \ref{hang}; note the additional edges. The new path $P'''$ is shown in bold.}
\label{farpath}
\end{center}
\end{figure}
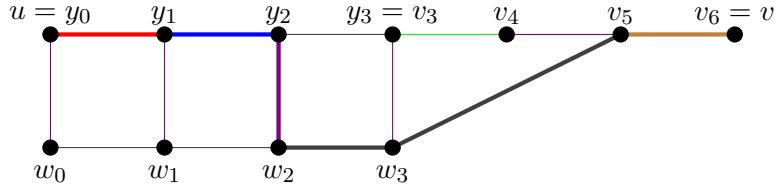

Now as $a_r > r$ and $a_{r-1} \geqslant r-1$ we have $\min \lbrace a_{r-1}, a_r - 2 \rbrace \geqslant r -1$, so there is always an edge of colour $a_r$ between $y_{r-1}$ and $w_{r-1}$. Consider the path $P'''$ obtained by using the first $r-1$ edges of $P_\mathbf{s'}$, the edge $y_{r-1}w_{r-1}$ and then the remaining edges of $P''$; we obtain
\[ P''':(uy_1 \cdots y_{r-1}w_{r-1} \cdots w_{a_r - 2}v_{a_r} \cdots v_d), \] 
where the first $r$ edges are of colours $a_1, \ldots, a_r$. We get the result by putting $x_k = y_k \in U_k$ for $1 \leqslant k < r$ and $x_k = w_{k-1} \in U_k$ for $r \leqslant k < a_r$, and setting $P_\mathbf{s} = P'''$.
\end{proof} 
We can now use this result to show that different sequences correspond to distinct vertices.
\begin{mycor}
\label{vert}
Let $\mathbf{s}$ and $\mathbf{t}$ be distinct, strictly increasing sequences from $[d]$. Then $u_\mathbf{s} \neq u_\mathbf{t}$.
\end{mycor}
\begin{proof}
Let $|\mathbf{s}| = q$ and $|\mathbf{t}| = r$, where we may assume $q \leqslant r$, and suppose that $u_\mathbf{s} = u_\mathbf{t}$. Write $\mathbf{s} = (a_1, \ldots, a_q)$ and $\mathbf{t} = (b_1, \ldots, b_r)$ and consider the paths $P_\mathbf{s}:(ux_1 \cdots x_{a_q} = v_{a_q} \cdots v_d)$ and $P_\mathbf{t}:(uy_1 \cdots y_{b_r} = v_{b_r} \cdots v_d)$ as described in Lemma \ref{seq}; note that $x_q = u_\mathbf{s} = u_\mathbf{t} = y_r$. Thus we may replace the $uy_r$-segment of $P_\mathbf{t}$ with the $ux_q$-segment of $P_\mathbf{s}$ to obtain a $uv$-walk in $G_S$ of length $d - (r-q)$, a contradiction if $q < r$. This shows that $U_q \cap U_r = \emptyset$ for $q \neq r$.

So suppose $q=r$ (so $x_r = y_r$) and for now that $a_r \neq b_r$; we may assume that $a_r < b_r$. Now replace the $ux_r$-segment of $P_\mathbf{s}$ with the $uy_r$-segment of $P_\mathbf{t}$ to obtain a new $uv$-path
\[P':(uy_1 \cdots y_r = x_r \cdots x_{a_r-1}x_{a_r}=v_{a_r} \cdots v_{b_r-1}v_{b_r} \cdots v_d), \]
where we note this is a path as both $P_\mathbf{s}$ and $P_\mathbf{t}$ pass sequentially through the pairwise disjoint sets $U_0, U_1, \ldots, U_d$. But both the $y_{r-1}y_r$ and $v_{b_r - 1}v_{b_r}$ edges on $P'$ are of colour $b_r$, contradicting Corollary \ref{distinct}.

Now note that in any case $x_{r-1} = (x_r)f_{a_r}$ and $y_{r-1} = (y_r)f_{b_r}$, so if $a_r = b_r$ we have $x_{r-1} = y_{r-1}$. An easy inductive argument (and the fact that $\mathbf{s} \neq \mathbf{t}$) shows that there exists some $p \leqslant r$ such that $x_i = y_i$ for $p \leqslant i \leqslant r$ but $a_p \neq b_p$, so we may apply the above argument to the sequences $\mathbf{s'} = (a_1, \ldots, a_p)$ and $\mathbf{t'} = (b_1, \ldots, b_p)$ to get a contradiction, thus proving the result.
\end{proof}

This allows us to prove the main result.

\begin{proof}[Proof of Theorem \ref{diam}]
If $d=0$ or 1 then the result is trivial, so suppose $d \geqslant 2$. $G_S[C]$ has diameter $d$, so there exist $u,v \in C$ such that $d_S(u,v)=d$; as before, assume that $X = [n]$ and $[d] \subseteq S$, and that there is a shortest $uv$-path using edges of colours $1, \ldots, d$ sequentially. There are $2^d$ strictly increasing sequences of elements from $[d]$ (including the empty sequence) and from Corollary \ref{vert} these correspond to $2^d$ distinct vertices (including $u$). Hence $|C| \geqslant 2^d$.

Note that there are $2^{d-1}$ strictly increasing sequences from $[d-1]$; adding $d$ to the end of such a sequence $\mathbf{s'}$ gives a strictly increasing sequence $\mathbf{s}$ from $[d]$. But then $(u_\mathbf{s'})f_d = u_\mathbf{s}$ so the edge $u_\mathbf{s'}u_\mathbf{s}$ is of colour $d$ for each such sequence $\mathbf{s'}$; as all vertices are distinct this gives us $2^{d-1}$ edges of colour $d$.
\end{proof}

\section{Extremal examples}
\label{extrsec}

In this section we construct a family of extremal examples, one for each $d$. Let $X_d := \lbrace u_1, \ldots, u_d \rbrace \cup \lbrace 0,1 \rbrace^d$ be a set of $2^d + d$ elements, and for each $v \in \lbrace 0,1 \rbrace^d$ and $S \subseteq [d]$, define $v_S$ to be the element of $\lbrace 0,1 \rbrace^d$ obtained by changing only the coordinates in $S$, i.e.\
\[ \lbrace i \in [d] \mid (v_S)_i \neq v_i \rbrace = S. \]
We will now define a set of bijections on $X_d$; for each $1 \leqslant i \leqslant d$, set $(u_j)f_{u_i} = u_j$ for all $j$ and $(v)f_{u_i} = v_{\lbrace i \rbrace}$ for all $v \in \lbrace 0,1 \rbrace^d$. We will also set $f_v = \iota$ for all $v \in \lbrace 0,1 \rbrace^d$.

Note that $(v_{\lbrace i \rbrace})_{\lbrace i \rbrace} = v$ and thus each map is an involution; as we also have $(x)f_x = x$ for each $x \in X_d$, we need only show that \eqref{maps} holds to prove that we have defined a kei $K$. Note that for any distinct $i,j \in [d]$ and $v \in \lbrace 0,1 \rbrace^d$,
\[ (v)f_{u_i}f_{u_j} = (v_{\lbrace i \rbrace})f_{u_j} = v_{\lbrace i,j \rbrace} = (v_{\lbrace j \rbrace})f_{u_i} = (v)f_{u_j}f_{u_i}, \]
and it follows that all the maps $(f_x)_{x \in X_d}$ commute. Hence \eqref{maps} reduces to the statement that $f_{(y)f_z} = f_y$ for all $y,z$; this is easily seen to be true.

Now consider the graph $G_K$; by construction, $G_K$ consists of $d$ isolated vertices and a coloured copy of the cube $Q_d$ (see Figure \ref{cubefig}). The cube has diameter $d$, size $2^d$ and contains $2^{d-1}$ edges of each colour $u_1, \ldots, u_d$.

\begin{figure}
\begin{center}
\begin{tikzpicture}
\draw [red] (0,0) -- (1.5,0);
\draw [red] (0,1.5) -- (1.5,1.5);
\draw [red] (0.7,0.6) -- (2.2,0.6);
\draw [red] (0.7,2.1) -- (2.2,2.1);

\draw [blue] (0,0) -- (0.7,0.6);
\draw [blue] (1.5,0) -- (2.2,0.6);
\draw [blue] (0,1.5) -- (0.7,2.1);
\draw [blue] (1.5,1.5) -- (2.2,2.1);

\draw [green] (0,0) -- (0,1.5);
\draw [green] (1.5,0) -- (1.5,1.5);
\draw [green] (0.7,0.6) -- (0.7,2.1);
\draw [green] (2.2,0.6) -- (2.2,2.1);

\draw [fill] (0,0) circle [radius=0.1];
\draw [fill] (1.5,0) circle [radius=0.1];
\draw [fill] (0,1.5) circle [radius=0.1];
\draw [fill] (1.5,1.5) circle [radius=0.1];

\draw [fill] (0.7,0.6) circle [radius=0.1];
\draw [fill] (0.7,2.1) circle [radius=0.1];
\draw [fill] (2.2,0.6) circle [radius=0.1];
\draw [fill] (2.2,2.1) circle [radius=0.1];

\node [below left] at (0,0) {000};
\node [below right] at (1.5,0) {100};
\node [left] at (0,1.5) {001};
\node [above left] at (0.8,0.6) {010};

\draw [red, fill=red] (-2.5,1.2) circle [radius=0.1];
\draw [blue, fill=blue] (-3,0.4) circle [radius=0.1];
\draw [green, fill=green] (-2,0.4) circle [radius=0.1];

\node [above] at (-2.5,1.2) {$u_1$};
\node [below left] at (-3,0.4) {$u_2$};
\node [below right] at (-2,0.4) {$u_3$};

\end{tikzpicture}
\end{center}
\caption{An example of a kei on $X_3$.}
\label{cubefig}
\end{figure}
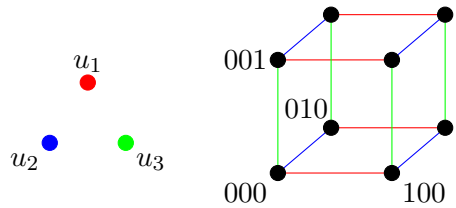

\section*{Acknowledgements}

This paper forms part of the author's Ph.D. thesis \cite{thesis}, supervised by Oliver Riordan; his assistance in preparing this paper is gratefully acknowledged. The thesis contains some applications of this result.

\end{document}